\documentclass{amsart}

\usepackage{amsfonts}
\usepackage{url}
\usepackage{hyperref} 
\hypersetup{backref,pdfpagemode=FullScreen,colorlinks=true}
\usepackage{amsmath}
\usepackage{amssymb}
\usepackage{amscd}
\usepackage{color}
\usepackage{amsthm}
\usepackage{bm}
\usepackage{indentfirst}
\usepackage[hmargin=3cm,vmargin=3cm]{geometry}
\usepackage[all, cmtip]{xy}
\usepackage{cite}

\numberwithin{equation}{section}
\newtheorem{thm}[equation]{Theorem} 
\newtheorem{prop}[equation]{Proposition}
\newtheorem{lma}[equation]{Lemma} 
 
\newtheorem{cor}[equation]{Corollary}

\theoremstyle{definition}

\newtheorem{df}[equation]{Definition}

\theoremstyle{remark}
\newtheorem*{pf}{Proof}

\newtheorem{rmk}[equation]{Remark}
\newtheorem{example}{Example}
\newtheorem{qtn}{Question}

\newcommand{\R}{{\mathbb{R}}}
\newcommand{\Z}{{\mathbb{Z}}}
\newcommand{\C}{{\mathbb{C}}}
\newcommand{\Q}{{\mathbb{Q}}}
\newcommand{\D}{{\mathbb{D}}}

\newcommand{\sm}[1]{C^\infty(#1)}

\newcommand{\delbar}{\overline{\partial}}

\newcommand{\til}[1]{\widetilde{#1}}
\newcommand{\wh}[1]{\widehat{#1}}

\newcommand {\ham} {\mathrm{Ham} (M, \omega)}
\newcommand {\hamcp} {\mathrm{Ham} (\mathbb{CP} ^{r-1}, \omega )}
\newcommand {\Cont} {\mathrm{Cont}}

\newcommand {\vM} {{(T^*)} ^{vert} \cM}
\newcommand{\om}{\omega}
\newcommand{\al}{\alpha}

\newcommand{\Om}{\Omega}

\newcommand{\cA}{\mathcal{A}}

\newcommand{\cD}{\mathcal{D}}
\newcommand{\cO}{\mathcal{O}}
\newcommand{\cE}{\mathcal{E}}
\newcommand{\cF}{\mathcal{F}}

\newcommand{\cH}{\mathcal{H}}

\newcommand{\cJ}{\mathcal{J}}

\newcommand{\cU}{\mathcal{U}}

\newcommand{\cQ}{\mathcal{Q}}
\newcommand{\cM}{\bm{M}}
\newcommand{\cP}{\bm{P}}
\newcommand{\cL}{\bm{L}}

\newcommand\vol{\operatorname{vol}}

\newcommand{\bP}{\mathbb{P}}
\DeclareMathOperator {\ho} {ho}
\DeclareMathOperator {\Proj} {P}
\renewcommand{\i}{\sqrt{-1}}

\DeclareMathOperator{\rank}{\mathrm{rank}}
\DeclareMathOperator{\ind}{\mathrm{ind}}
\DeclareMathOperator{\trace}{\mathrm{trace}}
\DeclareMathOperator{\image}{\mathrm{image}}
\DeclareMathOperator{\Sym}{\mathrm{Sym}}
\DeclareMathOperator{\Ham}{\mathrm{Ham}}
\DeclareMathOperator{\Aut}{\mathrm{Aut}}

\DeclareMathOperator{\Fred}{\mathrm{Fred}}
\DeclareMathOperator{\id}{\mathrm{1}}
\DeclareMathOperator{\coker}{coker}
\begin{document}







\title[K-theoretic invariants of Hamiltonian fibrations]{\textbf{K-theoretic invariants of Hamiltonian fibrations}}



\author{Yasha Savelyev}
\email{yasha.savelyev@gmail.com}
\address{CUICBAS, University of Colima, Bernal Díaz del Castillo 340,
Col. Villas San Sebastian, 28045 Colima, Mexico}

\author{Egor Shelukhin}
\email{egorshel@gmail.com, shelukhin@dms.umontreal.ca}
\address{Department of Mathematics and Statistics, University of Montreal, C.P. 6128, Succ. Centre-ville, Montreal, QC H3C 3J7, Canada}



\date{}


\begin{abstract}


We introduce new invariants of Hamiltonian fibrations with values in the
   suitably twisted K-theory of the base. Inspired by techniques of geometric
   quantization, our invariants arise from the family analytic index of a family
   of natural $Spin^c$-Dirac operators. As an application we give new examples
   of non-trivial Hamiltonian fibrations, that have not been previously detected
   by other methods. 
As one crucial ingredient we construct
a potentially new homotopy equivalence map, with a
certain naturality property, from $BU$ to the
space of index $0$ Fredholm operators on a Hilbert space, using
elements of modern theory of homotopy colimits.

 \end{abstract}

 \maketitle
 
 \tableofcontents



\section{Introduction and main results}




Since they were introduced in \cite{GLS}, Hamiltonian fibrations have been the subject of
active research in symplectic topology in the last few decades, 
see e.g. \cite{GLS,citeReznikovCharacteristicclassesinsymplectictopology,citePolterovichGromovs$K$-areaandsymplecticrigidity,LalondeMcDuffPolterovichTopRig,LalondeMcDuffFiber,KedraRestrictions,KedraMcDuff,citeSeidel$pi_1$ofsymplecticautomorphismgroupsandinvertiblesinquantumhomologyrings, citeMcDuffQuantumhomologyoffibrationsover$S2$,citeSavelyevQuantumcharacteristicclassesandtheHofermetric.,citeSavelyevBottperiodicityandstablequantumclasses.}
 and \cite{citeMcDuffSalamonIntroductiontosymplectictopology}. 
A Hamiltonian fibration is a topological fiber bundle $\cM$ over a base space\footnote{In this work a topological space will always be a paracompact Hausdorff topological space, with the associated category denoted by $Top$.} $B,$ with fiber given by a symplectic manifold $(M,\om),$ and structure group reduced to the group $\Ham(M,\om)$ of Hamiltonian diffeomorphisms of $(M,\om).$ 

One of the most basic examples of a Hamiltonian fibration is given by the
projectivization $\Proj(E) \to B$ of a complex vector bundle $E \to B$ of rank
$r$ over $B.$ Of course if $\Proj(E)$ is non-trivial as a complex projective-linear bundle, it is not immediate that it is non-trivial as a Hamiltonian fibration. Reznikov showed essentially the following by a suitable infinite-dimensional Chern-Weil theory in \cite{citeReznikovCharacteristicclassesinsymplectictopology} (since then a simpler, but ultimately equivalent, proof was found using characteristic classes of Hamiltonian fibrations - see \cite[Section 2.3]{McDuffLectures}).
 
\begin{thm}\label{thm:Reznikov-example} Let $E$ be a complex vector bundle on a topological space $B.$ If for each complex line bundle $U$ on $B,$ \[ch(E \otimes U) \neq ch(\underline{\C}^{r}),\] where $\underline{\C}^{r}$ denotes the trivial vector bundle of the same rank as $E,$ and $ch$ denotes the Chern character, then $\Proj(E)$ is non-trivial as a Hamiltonian fibration.\end{thm}

In this paper we study Hamiltonian fibrations further, developing new invariants and obtaining new results, inspired by geometric quantization. In particular, our construction yields the following lift of Reznikov's result to $K$-theory, providing new examples of non-trivial Hamiltonian fibrations.

\begin{thm}\label{thm:mini-example} Let $E$ be a complex vector bundle on a topological space $B.$ If for each complex line bundle $U$ on $B,$ $E \otimes U$ is not stably trivial, then $\Proj(E)$ is non-trivial as a Hamiltonian fibration.\end{thm}

This result is stronger than Theorem \ref{thm:Reznikov-example}, since the condition here can be rewritten as \[[E \otimes U] \neq [\underline{\C}^{r}]\] in $K^0(B),$ for all $U \in H^2(M,\Z),$ and the Chern character provides an isomorphism between the rational K-theory and the rational cohomology of $B$. This result is strictly stronger, as follows rather quickly from the fact that there exist complex vector bundles $E$ that are not stably trivial, and yet all their Chern classes vanish. See specific examples in Section \ref{section:examples}. Finally, to draw a classical analogy to this result, the projectivization of a complex vector bundle is trivial as a $PU(r)$-bundle if and only if the vector bundle is the twist by a line bundle of a trivial vector bundle. This is immediate from the central extension \[0 \to S^1 \to U(r) \to PU(r) \to 1.\]

Our new invariants of Hamiltonian fibrations are based on the observation that the theory of geometric quantization, when applicable, provides for each $k \in \Z_{>0}$ a homotopy-canonical (twisted, in general) family of elliptic differential operators on certain natural Hilbert bundles over $B.$ The analytic family index of this elliptic family furnishes the required invariants with values in the K-theory of the base. The family being homotopy-canonical follows from the contractibility of the space of $\om$-compatible almost complex structures on $M,$ a fact that is well-known, yet crucial, also in the theory of pseudo-holomorphic curves.

To describe these invariants in more detail we recall a few facts about geometric quantization. First of all, the basic object that we work with is a \emph{prequantization space}, which can be described as a $4$-tuple $\widehat {M}= (M, L, \nabla, \omega)$ for $L$ a complex Hermitian line bundle over a symplectic manifold $(M,\om)$ with a unitary connection $\nabla$ having curvature\footnote{\label{footnote1}This sign convention, while common in the quantization literature, is opposite to the one that is standard in the symplectic topology literature. In particular, the dual tautological line bundle $\cO(1) \to \C P^r$ with its canonical Chern connection, determined by the standard holomorphic and Hermitian structures, is a prequantization $\wh{\C P^r}$ of $\C P^r$ with the standard Fubini-Study symplectic form $\om_{FS}.$ To make this sign convention compatible with the usual statements on Hamiltonians and the automorphism group of the prequantization, one should use the sign convention $\iota_X \om = dH,$ for the Hamiltonian vector field $X$ of the Hamiltonian function $H$ on $M.$} $R(\nabla) = -i \om.$ Note that in order for such $(L,\nabla)$ to exist on a symplectic manifold $(M,\om),$ the cohomology class of the symplectic form must be representable by a class with integer coefficients. We call such symplectic manifolds {\em quantizable} and shall restrict our consideration to this class. 

In this setting one constructs for each $\om$-compatible almost complex structure $J,$ and $k \in \Z_{>0}$ a certain first order elliptic differential operator on the spaces of $L^{\otimes k}$-valued
$(0,*)$-differential forms on $M$, called the $Spin ^{c} $ Dirac operator $D^k$.
We  write $D ^{k} _{+} $ for the restriction of $D$ to differential forms
with even degree. Considering the formal difference of vector spaces, the index of $D^k _{+} $ can be thought to be a $\mathbb{Z} _{\geq 0} $-graded, virtual "Hilbert space" $$H ^{k} (\widehat{M})= [\ker D _{+}  ^{k}] -[\coker D _{+}  ^{k}].  $$ This is what we mean by \emph{quantization} in this paper, and the idea for this $K$-theoretic version of quantization is usually attributed to Bott (cf. \cite{BottQuantization}). Although we note that there are more detailed versions of quantization having various nice properties (cf. \cite{MaICM} and the references therein).


The main setup for family quantization in this paper is that of a continuous family of
prequantization spaces parametrized by a topological space $B,$ or more specifically, that
of {\em prequantum fibrations}: structure group $\cQ (\widehat{M})$ fibrations, where $\cQ(\widehat{M} )$ denotes the automorphism group of the prequantization space (not necessarily $id$ on the base). Each prequantum fibration is a line bundle over a unique Hamiltonian fibration. For example a suitable fiberwise blowup of a complex vector bundle furnishes its projectivization with a lift to a prequantum fibration. 

Our invariants are defined as follows. Given $k \in \Z_{>0},$ a (homotopy-canonical) choice of a continuous family of $\om$-compatible almost complex structures on the fibers of the Hamiltonian fibration canonically determines a family of $Spin^{c}$ Dirac operators, as above. The analytic index of this family of elliptic operators, as defined and studied by Atiyah-Singer in their celebrated series \cite{AtiyahSingerIV}, gives a K-theory class on the base $B.$ For each $k \in \Z_{>0},$ this class generalizes Bott's "virtual Hilbert space", does not depend on the choice of almost complex structures, and is consequently shown to be an invariant of the isomorphism class of the prequantum fibration. 
%
%
%


These K-theory invariants turn out to be rather powerful. In particular considering only the $H^{1}$ spaces (corresponding to $k=1$), we prove Theorem \ref{thm:mini-example}, and produce new examples of non-trivial prequantum and Hamiltonian fibrations. These invariants can also be used to prove that various other prequantum and Hamiltonian fibrations are non-trivial, for example those with fibers given by coadjoint orbits, but we defer these computations to further publications.

We deduce Theorem \ref{thm:mini-example} from the following stronger result on the level of classifying spaces, that is proven in Section \ref{sect:thm-exa-proof}. Set $\cQ (r) = \cQ (\widehat{\mathbb{CP}} ^{r-1} )$ for the group of automorphisms of the natural prequantization $\widehat{\mathbb{CP}} ^{r-1}$ of $\mathbb{CP}^{r-1}$ (see Footnote \ref{footnote1}). Let $B \cQ(r)$ be the classifying space in the category $Top$ of paracompact Hausdorff topological spaces for the group $\cQ(r),$ given by the Milnor construction \cite{citeMilnoruniversalbundles}. Let $\cH$ be a separable infinite-dimensional Hilbert space, and for $r \in \Z$ let $\Fred_r(\cH)$ be the space of Fredholm operators on $\cH$ of Fredholm index $r.$


\begin{thm}\label{thm:example}
There are natural maps $e,q, W_{r} $, with $e,q$ forming the sequence $$BU(r) \xrightarrow{e} B \cQ(r) \xrightarrow{q}
\Fred_r(\mathcal{H}),$$ with $W_{r} $ a weak equivalence and a homotopy commutative diagram:
\begin{equation*}
\xymatrix {BU (r) \ar [d]^{i _{r} }  \ar [r] ^{e} & B\cQ (r) \ar [d]
^{q} \\
BU \ar [r] ^{W _{r} } &  {\Fred _{r} (\mathcal{H})}.}
\end{equation*}
\end{thm}

The map $e$ above is induced by the canonical homomorphism $U (r) \to \cQ (r)
,$ and the map $i_r: BU(r) \to BU$ is the natural inclusion, and the weak equivalence $W_r$ is constructed in Section \ref{sect:infinity-categories}, as part of a proof of Theorem \ref{Proposition: homotopy equivalence commuting}. This theorem is a strengthening of the Atiyah-J\"{a}nich theorem, that is a key component in the proof of Theorem \ref{thm:example}. Theorem \ref{thm:example} implies additional results including the following (see Corollary \ref{cor:functors}).

\begin{thm}\label{thm:examplesimplified}
The map in complex K-theory induced by the natural inclusion map $BU (r) \to B\cQ (r) \equiv B \cQ (\widehat
{\mathbb{CP}}^{r-1} ) $, is surjective on K-theory.
\end {thm}
This strengthens a theorem of Spacil \cite{SpacilThesis}, based on Reznikov-type \cite{citeReznikovCharacteristicclassesinsymplectictopology} characteristic classes obtained by infinite-dimensional Chern-Weil theory for the strict contactomorphism group, that \[e^*: H^*(B \cQ(r); \R) \to H^*(B U(r); \R)\] is a surjection.

In general, a Hamiltonian fibration $\cM$ need not admit a lift to a prequantum fibration. This is measured by a class $\eta_{\cM} \in H^3(B,\Z)$ in the cohomology of the base $B,$ called the Dixmier-Douady class. Therefore, when this class does not vanish, a certain twisting for this class becomes necessary to define our K-theoretic invariants of Hamiltonian fibrations (in fact even in the case of vanishing Dixmier-Douady class, the twisted invariant gives a more conceptual picture). Indeed, we are able to produce a canonical invariant of $\cM$ with values in {\em twisted K-theory}, and prove Theorem \ref{thm:twistedmain}, that is a partial analogue of Theorem \ref{thm:example} in the twisted setting.


In Section \ref{section:discussion} we prove Proposition \ref{prop: cw quantized}, inspired by family quantization and the family Atiyah-Singer theorem, that the Calabi-Weinstein invariant \cite{WeinsteinCalabi} of loops of automorphisms of prequantization spaces has image in $\frac{1}{n+1} \Z,$ an integrality statement that was open since the work of Banyaga and Donato \cite{BanyagaDonato}. Consequently, we show in Proposition \ref{Proposition:DD class torsion} that the Dixmier-Douady class of a Hamiltonian fibration (with quantizable fibers) is always $(n+1)$-torsion.

\medskip

{\bf Push-forward in complex K-theory.} In the case of smooth fibrations, one can interpret our construction as the push-forward of a natural (twisted in general) line bundle over the total space $\cM$ of a Hamiltonian fibration corresponding to a (perhaps defined only locally in $B$) prequantum fibration over $\cM.$ In the case when the prequantum fibration is defined globally, and hence the line bundle is not twisted, the push-forward (also known as fiber-integration, Umkehr, or Gysin) map in complex K-theory can be defined by classical topology using the Thom isomorphism, Bott periodicity, and embedding techniques (see e.g. \cite[Chapters IV,V]{KaroubiBook}). For the recently studied push-forward in twisted K-theory we refer to \cite{CareyWang}. We note that the computation of our invariants seems to require the analytic definition which we present in full detail in this paper, and does not seem to be evident from the classical construction of the push-forward.

As a remark on the literature, we note that family $Spin^c$ quantization was considered for different purposes by Zhang \cite{ZhangReductionFamilies}, and deformation-quantization of symplectic fibrations was studied in \cite{OlgaKr}. Examples of the relation between quantization and symplectic topology appear in \cite{PolterovichNoise,PolterovichUnsharpness,DonaldsonApprox}.

\medskip


{\bf Organization of the paper.} In Section \ref{sect:prelim} we recall the preliminary material necessary for our arguments and constructions: prequantum fibrations in Section \ref{sect:prequant-fib}, $Spin^c$-Dirac operators in Section \ref{sect:Spinc-Dirac}, and analytic family index in Section \ref{sect:family-index}. Section \ref{sect:quantization-Bott} introduces our main $K$-theoretic invariants. Section \ref{sect:thm-exa-proof} proves Theorem \ref{thm:example}. Section \ref{section:examples} discusses new examples of non-trivial Hamiltonian and prequantum fibrations. Section \ref{sect:super-twisted} introduces a twisted version of our invariants and proves Theorem \ref{thm:twistedmain}. Section \ref{section:discussion} discusses topics related to quantization of prequantum fibrations, and proves Proposition \ref{prop: cw quantized} and Proposition \ref{Proposition:DD class torsion}. Finally, Section \ref{sect:infinity-categories} proves Theorem \ref{Proposition: homotopy equivalence commuting}, which is an indispensable step in the proof of Theorems \ref{thm:example} and \ref{thm:twistedmain}, by means of the theory of infinity categories.

\section{Preliminaries}\label{sect:prelim}

\subsection{Prequantum fibrations}\label{sect:prequant-fib}

To set up geometric quantization on a symplectic manifold
$(M_0,\om_0)$ one requires that there be a complex Hermitian line
bundle $(L_0,h_0)$ over $M_0$ with a unitary connection $\nabla_0$ having
curvature $R(\nabla_0) = -i \om_0.$ This is equivalent to the vanishing of the image of
the class of the symplectic form in $H^2(M_0,\R/\Z)$ 
\cite[Proposition 8.3.1]{woodhouse}. We call such symplectic manifolds
{\em quantizable}, and the symplectic manifolds we consider in this
paper are all of this kind. A given prequantization of $(M_0, \omega_0)$ will be denoted by
$\widehat{M}_0 $. Let $\wh{p}: L_0 \to M_0$ be the projection. The group of automorphisms $\cQ(\wh{M}_0)$ of $\wh{M}_0$ is defined to be the identity component of the subgroup of those diffeomorphisms $\wh{\phi}$ of the total space of $L_0$ that are bundle maps, that is $\wh{p} \circ \wh{\phi} = \phi \circ \wh{p}$ for a diffeomorphism $\phi$ of $M_0,$ restrict to unitary maps on the fibers of $L_0 \to M_0,$ and preserve the connection $\nabla_0$: \[\phi_* \nabla_0 = \nabla_0.\]
In an alternative equivalent description, one requires that there be a principal $S^1$-bundle $P_0$ over $M_0$ with connection one-form $\al_0$ with curvature $d\al_0 = -\til{\om}_0,$ where $\til{\om}_0$ is the lift of $\om_0$ to $P_0$ by the projection map $p: P_0 \to M_0$ coming from the bundle structure. In this description, the automorphism group is given by $\cQ(P_0)=\Cont_0(P_0,\al_0),$ the identity component of the group of strict contactomorphisms of $(P_0, \al_0),$ those diffeomorphisms of $P_0$ that preserve the form $\alpha_0$. Note that $\cQ(P_0) \cong \Aut_0(L_0,h_0,\nabla_0) \cong \cQ(\wh{M}_0).$ It is standard (see \cite[Section 1.13]{KostantPrequant}, and \cite{BanyagaRegularContact}) that $\cQ(P_0)$ is a central $S^1$-extension of $\Ham(M_0,\om_0):$ \begin{equation}\label{eqn:Quant-Ham-extension} 1 \to S^1 \to \cQ(P_0) \xrightarrow{pr} \Ham(M_0,\om_0) \to 1.\end{equation} Here the projection $pr(\phi) \in \Ham(M_0,\om_0)$ of $\phi \in \cQ(P_0)$ is given by $pr(\phi):M_0 \to M_0,$ $x \mapsto p \circ \phi (y)$ for any $y \in p^{-1}(x).$ 


We would like to consider the same situation of prequantization in families - hence we would like a notion of prequantization of a Hamiltonian fibration (cf. \cite{GLS}). We recall one of the several equivalent definitions of a Hamiltonian fibration. 

\begin{df}[Hamiltonian fibration]\label{Definition: Hamiltonian
   fibration} A Hamiltonian fibration is a fiber bundle $M
   \hookrightarrow \cM \xrightarrow{\pi} B,$ where $B \in Top$  and $(M,\om)$ is a symplectic manifold, with structure group reduced to $\ham$.
   \end{df}

\begin{df}[Prequantum fibration]\label{Definition: prequantum fibration}
For a Hamiltonian fibration $\pi: \cM \to B,$ with fiber $(M,\om),$ $(\{\om_b\}_{b \in B})$
denoting the fiberwise symplectic structure, a \emph{prequantum lift}
is a Hermitian line bundle $\cL \to \cM$, with a continuous choice of  unitary connection
$\nabla_b$ on each fiber $\cL_b \to \cM_b$ over $b \in B$ with curvature $R(\nabla_b) \in \Om^2(\cM_b,i\R)$
equal to $-i \om_b,$ and moreover the structure group of $\cL \to B$ is $\cQ(\wh{M}),$ for a given fixed prequantization $\wh{M} = (L,\nabla,M,\om)$ of $(M,\om),$ compatible with the natural homomorphism $\cQ(\wh{M}) \to \Ham(M,\om).$ A Hamiltonian fibration with a given prequantum lift will be called a \emph{prequantum
fibration}, and a given prequantum lift of $\cM \to B$ will be denoted
by $\widehat{\cM} $.
\end {df}

Alternatively, a prequantum fibration is a bundle with model fiber $p: P_0 \to M_0,$ a prequantization (principal $S^1$-bundle) of a symplectic manifold, with structure group $\cQ(P_0)=\Cont_0(P_0,\al_0).$ 

Given a Hamiltonian fibration $\cM \to B$ with model fiber
$(M_0,\om_0)$ that has a prequantization $(L_0,h_0,\nabla_0)$ (that is
$[\om_0] = 0$ in $H^2(M_0,\R/\Z)$), the existence of a prequantum lift
$\widehat{\cM} $ of $\cM$ is controlled by a certain class in $H^2(B,\R/\Z) \cong H^3(B,\Z)$ (see \cite[Chapter 4]{BrylinskiBook} for a detailed discussion, and further references). This class is called the Dixmier-Douady class of the fibration $\cM \to B.$ Briefly, it is given as the image under the connecting map $H^2(M_0, \R/\Z) \to H^3(M_0,\Z)$ (associated to the coefficient exact sequence $0 \to \Z \to \R \to \R/\Z \to 0$) of a class in $H^2(M_0, \R/\Z)$ represented by the following $S^1$-valued \v{C}ech $2$-cocycle. Let $\cU = \{U_a\}$ a cover of $B.$ Consider $c_{ijk} = \wh{g}_{ij} \wh{g}_{jk} \wh{g}_{ki}$ where $\wh{g}_{ab}$ is an arbitrary lift of the structure map $g_{ab}: U_a \cap U_b \to \Ham(M_0,\om_0)$ of $\cM \to B$ with respect to a trivialization over $\cU$ to a map $\wh{g}_{ab}: U_a \cap U_b \to \cQ(\wh{M}_0).$ Note that since $pr(c_{ijk}) = g_{ij} g_{jk} g_{ki} = \id,$ by central extension \eqref{eqn:Quant-Ham-extension}, $c_{ijk}$ defines an $S^1$-valued $2$-cochain, which is rather easily checked to be a cocycle. In Section \ref{section:discussion} we prove the following statement which appears to be new.

\begin{prop}\label{Proposition:DD class torsion}
The Dixmier-Douady class of any Hamiltonian fibration with quantizable fibers is torsion.
\end{prop}


\subsection{The $Spin^c$-Dirac operator} \label{sect:Spinc-Dirac}
We briefly recall the construction and main properties of the $Spin^c$ Dirac operator. For more details see e.g. \cite{DuistermaatLefschetzBook, MaICM, MaMarinescuBook,GinzburgGuilleminKarshon}. Given a symplectic manifold $(M,\om),$ let $L$ be a Hermitian line bundle over $M$ with a Hermitian connection $\nabla^L.$ We later consider the case when $(L,\nabla^L)$ is a prequantization of $(M,\omega).$

Then a choice of an $\om$-compatible almost complex structure $J$ on $M$ canonically determines the $Spin^c$-Dirac operator as follows. Consider $(TM,J)$ as a complex vector bundle with Hermitian structure given by $(\om,J)$ - in particular we consider the $J$-invariant Riemannian metric  $g_J(\cdot,\cdot) = \om(\cdot,J\cdot).$ 
We remark that $(TM,J) \cong TM^{(1,0)},$ the $\i$-eigenbundle of $J \otimes 1$ acting on the complexification $TM \otimes_{\R} \C,$ and $TM^{(0,1)} \cong \overline{TM}^{(1,0)}$ is the $(-\i)$-eigenbundle of this action. Then $TM \otimes_{\R} \C \cong TM^{(1,0)} \oplus TM^{(0,1)}$ Similarly we define $T^*M^{(1,0)}, T^*M^{(0,1)}.$ 
This Hermitian vector bundle carries a canonical Hermitian connection $\nabla^{Ch}$, which is uniquely characterized by the properties $\nabla^{Ch} \om =0, \nabla^{Ch} J = 0,$ and $T_{\nabla^{Ch}}^{(1,1)} = 0.$  This induces a connection $\nabla^{Ch,n}$ on $K = T^*M^{(n,0)}.$ Denote \[E=E_{(TM,J)}=\Lambda(T^*M^{(0,1)}) = \bigoplus_{0 \leq q \leq n} \Lambda^{(0,q)} (T^*M).\]

\begin{rmk}
In general any other $J$-invariant Riemannian metric $g$ on $M,$ and any Hermitian connection on $K$ would define a $Spin^c$-Dirac operator. However, in what follows we use the above specific metric $g_J$ and Hermitian connection $\nabla^{Ch,n},$ because they are determined canonically by the pair $(\om, J).$
\end{rmk}

Given the connections $\nabla^L$ and $\nabla^{Ch}$ we construct the
$Spin^c$-Dirac operator \[D ^{1}  (L, J ): \Gamma(L \otimes E; M) \to \Gamma (L \otimes E; M)\] acting on the smooth sections of $L \otimes E$ over $M$ as follows. 

One first shows that $E$ and hence $L\otimes E$ is a Clifford module over the bundle of complexified Clifford algebras $C(TM) \otimes_{\R} \C.$ The Clifford action of $v \in T^{(1,0)} M, \overline{v} \in T^{(0,1)} M$ on $E$ is \[c(v) = \sqrt{2}\cdot \overline{v}^* \wedge, \] \[c(\overline{v}) = -\sqrt{2}\cdot  \iota_{\overline{v}},\] where  $\overline{v}^* \in T^* M^{(0,1)}$ is the metric dual of $v.$

One further shows that $\nabla^{Ch,n}, \nabla^L$ and the Levi-Civita
connection of $g_J$ give a connection $\nabla^{Cl}$ on $L \otimes E$
compatible with the Clifford module structure (a Clifford connection).
Then for $\eta \in \Gamma(L \otimes E)$ and $x \in M$ we define \[D
^{1}  (L, J )(\eta)_x = \sum_{j=1}^{2n} e_j \cdot \nabla^{Cl}_{e_j}(\eta)_x,\] where $(e_j)_{j=1}^{2n}$ is an orthonormal local frame in $TM_x$ and for $e \in TM_x,$ $e \cdot = c(e)$ denotes the Clifford module action. The operator $D ^{1}  (L, J )$ is an elliptic first order operator that is self-adjoint with respect to a natural inner product on $\Gamma(L \otimes E).$ Put \[E^q=\Lambda^{(0,q)} (T^*M).\] Then under the splitting $E =  E^{-} \oplus E^{+},$ of $E$ into \[E^{+} = \displaystyle\bigoplus_{q \; \text{even}} E^q \] and \[E^{-} = \displaystyle\bigoplus_{q \; \text{odd}} E^q ,\] the operator $D
^{1}  (L, J )$ splits as   

\[D^{1}  (L, J ) = \begin{pmatrix}
 0 & D
^{1} _{+}   (L, J ) \\ D
^{1} _{-}   (L, J ) & 0 
\end{pmatrix},\] for operators \[D
^{1} _{+}   (L, J ): \Gamma(L \otimes E^+) \to \Gamma(L \otimes E^-),\] \[D
^{1} _{-}   (L, J ): \Gamma(L \otimes E^-) \to \Gamma(L \otimes E^+).\]


For an integer $k \geq 1$ put \[D ^{k}(L, J) = D^1(L^k,J),\] \[D ^{k}_{\pm}(L, J) = D^1_{\pm}(L^k,J),\] where $L^k$ is endowed with the Hermitian structure and connection induced from those on $L.$ When considering a specific prequantization $\wh{M} = (L,\nabla,M,\om)$ of $(M,\om),$ we use the notation $D^{k}(\wh{M}, J), D ^{k}_{\pm}(\wh{M}, J)$ for $D^{k}(L, J), D ^{k}_{\pm}(L, J).$

We require the following calculation of this operator for K\"{a}hler manifolds. Given a holomorphic Hermitian line bundle $L$ over a K\"{a}hler manifold $M,$ consider the Chern connection $\nabla^L$ on $L,$ and the K\"{a}hler structure $(\om,J)$ on $M.$ Then there is a natural $\delbar$-operator $\delbar_{L\otimes E}:\Gamma(L\otimes E) \to \Gamma(L\otimes E).$ The inner product on $\Gamma(L\otimes E)$ induces the adjoint operator $\delbar^*_{L\otimes E}:\Gamma(L\otimes E) \to \Gamma(L\otimes E).$ The sum \[\delbar_{L\otimes E} +\delbar^*_{L\otimes E}:\Gamma(L\otimes E) \to \Gamma(L\otimes E)\] is sometimes called the Dolbeault-Dirac operator. Then we have the following identity for the $Spin^c$-Dirac operator (see \cite[Theorem 1.4.5]{MaMarinescuBook}, observing that the relevant torsion tensor vanishes in the K\"{a}hler case). 

\begin{prop}\label{Proposition: Dirac Dolbeault}
\[D^{1}  {(L, J )} = \sqrt{2}(\delbar_{L\otimes E} + \delbar^*_{L\otimes E}).\]
\end{prop}

Therefore \[\ker D
^{1}  (L, J ) \cong H^*(M;\cO(L)),\] the cohomology of the sheaf $\cO(L)$ of holomorphic sections of $L.$ In fact \[\ker D
^{1}  (L, J )|_{\Gamma(L \otimes E^q)} \cong H^q(M;\cO(L)).\] 

\begin{cor}\label{Corollary:Diraccomputationforprojective}
When $M=\bP (V)$ for a Hermitian complex vector space $V$ with the standard
K\"{a}hler structure and $L = \cO(1),$ the dual tautological bundle, endowed as above with the Chern connection, \[\ker D^1_-{(L,J)}
\cong \mathrm{coker} D_+^1(L,J) = 0\] and \[\ker D_+
^{1}  (\widehat{M} , J ) = \ker D
^{1} _{+}   (\widehat{M}, J ) \cong V^*,\]
where the isomorphism is canonical. 

\end{cor}
\begin{proof}  
 All the non-zero sheaf cohomology of $\cO(1)$ vanishes (cf. \cite[Corollary 9.1.2]{KempfBook}, \cite[p. 156]{GriffithsHarris}). On the other hand $H ^{0} (M, \cO
  (L)) $ is the space of global holomorphic sections of $L,$ which by elementary algebraic geometry
  is canonically identified with the dual space of $V$.
\end{proof}

\begin{rmk}
If we replace $L$ by $L^k = \cO(k)$ in Corollary \ref{Corollary:Diraccomputationforprojective}, then the conclusion is the same with the sole difference that $\ker D
^{k} _{+}   (L, J ) = \ker D^{1} _{+}   (L^k, J ) \cong \Sym^k V^*.$
\end{rmk}

\subsection{The family analytic index} \label{sect:family-index}


Let $\cM \to B$,  be a locally trivial
fibration whose fibers are smooth manifolds, and whose structure group is reduced to the diffeomorphism group of the general fiber. Let $\cE, \cF \to \cM$ be two complex vector bundles, whose
restrictions to the fibers of $\cM \to B$ are smooth complex vector
bundles. For each $b \in B,$ let $\mathcal{D}_{b}:  \Gamma(\cE_b; \cM_b) \to \Gamma(\cF_b; \cM_b)$ be an order $1$ elliptic operator between the smooth
sections of the two vector bundles over the manifold $\cM_b.$ 
The family $\mathcal{D} \equiv \{\mathcal{D} _{b} \}$ is required to be
continuous with respect to the natural Frechet topology induced by the $C
^{\infty} $ topology on the bundles $\Gamma(\cE; \cM)$, $\Gamma(\cF;
\cM)$, with corresponding fibers over $b$: $\Gamma(\cE_b; \cM_b)$,
$\Gamma(\cF_b; \cM_b)$, see Atiyah-Singer \cite{AtiyahSingerIV}.
We shall call this family $\mathcal{D} = \mathcal{D} (\cE, \cF)$ of operators an \emph{Atiyah-Singer family}. 

We show how to associate to this situation an index in $K_0(B):=[B,\Fred
(\mathcal{H})],$ where $\Fred (\mathcal{H}) \simeq BU \times \mathbb{Z} $
denotes the space of Fredholm operators on a separable
infinite-dimensional Hilbert space\footnote{This is what we call a Hilbert space in this paper.}.
Assume that the
fibration $\cM$ is endowed with a fiber-wise smooth measure, and the bundles
$\cE \to \cM$ and $\cF \to \cM$ are Hermitian vector bundles. This endows the
fiber-wise smooth section bundles $\Gamma(\cE;\cM)$, $\Gamma(\cF;\cM)$, 
with a natural fiber-wise inner product.
 Indeed consider the  bundles of Hilbert spaces
$\cH_1=L^2_1(\cE;B)$ and $\cH_0=L^2(\cF;B)$ over $B$ given by the completion of
the corresponding fiber-wise section spaces. The family $\cD$ of elliptic
operators induces a Fredholm map \[[\cD]: \cH_1 \to \cH_0\] of  Hilbert-bundles
over $B.$ Since by Kuiper's theorem \cite{Kuiper} the unitary group of a Hilbert space is
contractible, the Hilbert bundles $\cH_1$ and $\cH_0$ can be trivialized, and
moreover these trivializations are \emph{homotopy canonical}, which
means, throughout the paper, that the space of choices (in this case of
trivializations) is contractible.
Choose unitary trivializations \[\Phi_1: \cH_1 \to \cH \times B,\;\Phi_0: \cH_0 \to \cH \times B.\] Then \[\Phi_0 \circ [\cD] \circ (\Phi_1)^{-1}: \cH \times B \to \cH \times B\] is a Fredholm bundle-map from the trivial $\cH$-bundle over $B$ to itself, and hence gives a map \[ f_{\Phi_0 \circ [\cD] \circ \Phi_1}: B \to \Fred(\cH).\] Define the family analytic index of the family $\cD$ as the homotopy class \[[f_{\Phi_0 \circ [\cD] \circ \Phi_1}] \in [B, \Fred(\cH)]\] which is well-defined, because the trivializations were homotopically canonical. 

 As mentioned above $\Fred(\cH)$ is well known to be weakly equivalent to $BU \times \Z.$
  We shall require a stronger version of this statement, where we ask that the
  weak equivalence homotopically commute with certain natural maps $p_r: BU(r)
  \to \Fred_0(\cH) = \ind^{-1}(0)$ (for the Fredholm index map $\ind:
  \Fred(\cH) \to \Z$), and the natural maps $i_r: BU(r) \to BU.$ Consider the
  model of $BU(r)$ as the Grassmannian $Gr(r,\cH)$ of $r$-planes in a fixed
  infinite-dimensional separable Hilbert space $\cH.$ This gives us the map
  $\widetilde{p}_{r}:   BU(r) \to \Fred_r(\cH) = \ind^{-1}(r)$ as follows. To an $r$-plane $H \subset \cH$ we
  associate a Fredholm operator $\widetilde{p}_{r}  (H): \cH \to \cH$ sending $H$ to $0$, and the orthogonal complement $H^\perp$ of $H$ isometrically onto $\cH.$ The choice of the latter isometry is determined by a section of the bundle of unitary maps $U(\tau^\perp,\underline{\cH})$ from the orthogonal complement $\tau^\perp$ of the tautological bundle $\tau$ over $Gr(r,\cH)$ (in the trivial $\cH$-bundle $\underline{\cH}$), to the trivial bundle $\underline{\cH}.$ This section exists and is homotopically-canonical, since by Kuiper's theorem the fiber of $U(\tau^\perp,\underline{\cH})$ is contractible. We also have index shift maps
\begin{equation} \label{eq:indexshift}
s _{r}: \Fred _{r}(\cH)  \to \Fred _{0}(\cH),
\end{equation}
defined by sending an operator $O$  to the operator
\begin{equation*}
   \cH \xrightarrow{O \oplus 0} \cH \oplus \mathbb{C}^{r}
   \xrightarrow{iso _{r} } \cH,
\end{equation*}
with the second map a fixed isometry.
Define $$p _{r} = s _{r} \circ \widetilde{p}_{r}. $$

An important technical role in this paper is played by the following strengthening of the Atiyah-J\"{a}nich theorem, that we prove in Section \ref{sect:infinity-categories}.

\begin{thm}\label{Proposition: homotopy equivalence commuting}
There exists a weak equivalence $W: BU \to \Fred_0(\cH)$ such that $W \circ i_{r} \simeq p _{r} $.
\end{thm}
The shift maps $s _{r} $ are homotopy equivalences, whose homotopy
inverses $s ^{-1}_{r}  $ can be constructed by hand. In fact these
maps are just representatives of the canonical shift maps in $\pi_0$.
That is $\pi _{0} (\Fred (\cH))= \mathbb{Z} $ as a group, with with the isomorphism
given by the index map. In particular fixing an element $e _{r} \in
\Fred (\cH) $ with index $-r$, we have homotopy equivalences:
\begin{align*}
& \overline{s}_{r}: \Fred _{r}  (\cH) \to \Fred _{0} (\cH) \\
& \overline{s}_{r} (O) = O \cdot e _{r},
\end{align*}
with $\cdot$ the multiplication in the $H$-space $\Fred (\cH)$. The
homotopy inverse is given by multiplication with the adjoint operator of $e
_{r}$.
The maps $s _{r} $ are easily seen to be homotopy equivalent to
$\overline{s}_{r}  $, using Kuiper's theorem.
We define 
\begin{equation} \label{eq:Wr}
W _{r} = s _{r}^{-1} \circ W.
\end{equation}

Then it is an easy consequence of Theorem \ref{Proposition: homotopy equivalence commuting} that 
\begin{equation} \label{eq:Wr-pr}
W_{r} \circ i_r \simeq \til{p}_r.
\end{equation}

\section{Quantization of prequantum fibrations} \label{sect:quantization-Bott}

In \cite{BottQuantization} Bott has suggested to use the index virtual space
$[\ker D_+ (\widehat{M}, J )] -[\coker D_+ (\widehat{M},J )]$ as the
quantization for the symplectic manifold $(M,\om)$ (cf. \cite{MaICM}).
The main tool of this paper is a family version of such quantization,
which becomes naturally K-theory valued.

Let there be given a Hamiltonian fibration $\cM \to B$   with a prequantum lift $\widehat{\cM} = (\cM,\cL,\nabla,\{\om_b\}_{b \in B})$ (recall Definition \ref{Definition: prequantum fibration}). 
%
%
Consider the associated bundle $\cJ_{\cM}
\to B$ with structure group $\Ham(M),$ whose fiber $(\cJ_{\cM})_b$ over
$b \in B$ is the space $\cJ_{\cM_b}$ of $\om_{\cM_b}$-compatible almost
complex structures on $\cM_b.$ This bundle has contractible fibers and
therefore has global sections. Pick a global section $\{J_b\} \in
\Gamma(\cJ_{\cM}; B).$ Denote by $T ^{vert} \cM $ the vertical tangent
bundle to the fibration $\cM \to B,$ consider it as a complex
vector bundle using $\{J_b\},$ and let ${(T^*)} ^{vert} \cM$ be its dual bundle. Denote by $E$ the corresponding exterior algebra bundle $\Lambda(\vM^{(0,1)}),$ and put
$E^+ = \Lambda^{even}(\vM^{(0,1)})),$  $E^- =
\Lambda^{odd}(\vM^{(0,1)})).$ To the data \[(\widehat{\cM}  \to B,
\{J_b\},k), \; k \in \Z_{>0}\] there corresponds a natural Atiyah-Singer
family ${D_+ ^{k}}  (\widehat{\cM},   \{J_b\} )$, that is the family of
operators:
\begin{equation} \label{eq:Dirac}
{D_+^{k}}  (\widehat{M} _{b}  ,  \{J_b\} ):
\Gamma(\mathcal{E} ^{+} _{b} ) \to
\Gamma(\mathcal{E} ^{-} _{b}  )
\end{equation}
where
$\mathcal{E} ^{+} =  \cL^k \otimes E^+$, and $\mathcal{E} ^{-}
= \cL^k \otimes E^-$. 
%
%

\begin{prop}\label{Propostion: K-theory}
   The family analytic index \[H^k(\widehat{\cM} )= [\ind(D_+^k
   (\widehat{\cM}, \{J _{b} \} ))] \in K^0(B)\]
of the elliptic Atiyah-Singer family, with  $\ind(D^k_+
   (\widehat{\cM}, \{J _{b} \} ))$ denoting the (homotopy-canonically defined) induced map from $B$ to $\Fred
   (\cH)$,
 does not depend on the choice of
the family $\{J_b\}$, and moreover is invariant under isomorphisms of prequantum fibrations. 
\end{prop}

\begin{pf} 
The proof of this claim follows directly from the homotopy invariance
of the analytic family index, and the observation that if $\Phi:
\widehat{\cM} 
\to \widehat{\cM}'$ is an isomorphism of prequantum fibrations lifting an isomorphism $\phi: \cM \to {\cM}'$ of Hamiltonian fibrations, then $\Phi^*D_+^k
(\widehat{\cM}', \{J _{b} \} )= D_+ ^{k}  (\widehat{\cM},  \phi ^{*}\{ J _{b} \}  )$ and hence the isomorphism invariance follows from the independence on the family of almost complex structures.
\end{pf}

\section{Application to $\Cont_0(S^{2r-1},\alpha_{std})$}\label{sect:thm-exa-proof}
Consider $\C^r$ with its standard complex structure and Hermitian
metric. Note that the unitary group $U(r)$ of this Hermitian structure
embeds into $\cQ(r) \equiv \cQ (\widehat{\C P} ^{r}).$ This follows for example from the fact that the Chern connection on $\cO(-1) \to \C P^{r-1} = \bP(\C^r)$ is invariant under the action of $U(r)$ on the total space $\C^r\setminus \{0\}$ of $\cO(-1),$ and hence the same statement is true about the induced action on $\cO(1) \to \C P^{r-1}.$ This implies that by extending the
structure group, every complex (Hermitian) vector bundle $V$ over $B$ can be considered as a prequantum fibration. Specifically, we associate to $V$ the prequantum fibration $\cO(1)_{\bP(V)} \xrightarrow{\pi} \bP(V) \to B,$ where over a point $b \in B,$ the corresponding projection $\pi_b : \cO(1)_{\bP(V_b)} \to \bP(V_b)$ is isomorphic to the dual tautological line bundle $\cO(1) \to \C P^{r-1},$ that is the complex line bundle associated by the $1$-dimensional $\C^*$-representation $z \mapsto z^{-1}$ to the principal $\C^*$-bundle $\C^r \setminus 0 \to \C P^{r-1}$ sending a point to the line it spans. The connection on each fiber is the Chern connection. This determines a prequantum lift $\wh{\bP(V)}$ of the Hamiltonian fibration $\bP(V) \to B.$

%
Given a finite-rank complex vector bundle $E$ over $B$, we consider the prequantum fibration $\wh{\bP(E^*)}.$ We denote $\cL_E = \cO(1)_{\bP(E^*)}$ and $\cM_E = \bP(E^*).$ We claim that the quantization procedure $E \to H^1(\cL_E)$ recovers the class of $E$ in the $K$-theory of $B.$ We prove Theorem \ref{thm:example} directly, since the prequantum analogue of Theorem \ref{thm:mini-example} is an immediate special case, while the proof is nearly identical. Moreover Theorem \ref{thm:mini-example} is a quick corollary, as explained in Section \ref{section:examples}.

%
\begin{proof} [Proof of Theorem \ref{thm:example}]
 The map $e$ is the map induced by the canonical homomorphism $U
 (r) \to \cQ (r)$.  
The map $W _{r} $ is as in \eqref{eq:Wr}.
The map $q: B \cQ(r) \to \Fred (\mathcal{H})$ is the (homotopically-canonically defined) index map $ind
(D_+ ^{1} (\widehat{P} _{B \cQ (r)},\{J_b\}))$, where $\widehat{P} _{B \cQ (r)} $ denotes the
   universal $\widehat{\C P} ^{r-1}  $ fibration over $B \cQ (r)$ and $\{J_b\}_{b \in B \cQ(r)}$ is a section of the bundle of compatible almost complex structures on the fibers of the corresponding Hamiltonian $\C P^{r-1}$ fibration.
%
%
We now show that
   $q \circ e$ and $W_r \circ i _{r} $ are homotopy equivalent.
Let $E \to BU (r)$ denote the tautological rank $r$ Hermitian vector
bundle (recall that $BU (r)$ for us is the space of $r$-planes in
$\mathcal{H}$).  
We note that the corresponding prequantum fibration $\cL_E \to \cM_E
\to BU (r)$
comes equipped with an almost complex structure $\{j_x\} \in \cJ_{\cM_E}$ such that
on each fiber over a point $x \in BU(r),$ $j_x$ is a standard complex structure on
$\cM_E = \bP(E^*)$ coming from a standard complex structure on $E^*.$ Note that in this case \[\ker D_+^1 (\widehat{\C P} ^{r-1} _{x} , j _{x}) - \coker D_+^1
   (\widehat{\C P} ^{r-1} _{x}, j _{x} )  =
H^0(\bP(E^*_x); \cO(1)_{\bP(E^*_x)}) = E_x,\] where the isomorphisms are canonical,
by Corollary \ref{Corollary:Diraccomputationforprojective}, and there is no cokernel
hence the proof is finished by Kuiper's theorem. More
specifically we may conclude by Kuiper's theorem that the fibration $\mathcal{E} \to BU (r)$, with fiber the space of
Fredholm operators on $\mathcal{H}$ with vanishing cokernel and kernel
$E$, has contractible fibers. Next we observe that $q \circ e$ and
$\til{p}_r \simeq W_r \circ \iota_r$ are both sections of this fibration, and so must be
homotopic. 
\end {proof}

We immediately obtain the following consequence.

\begin{cor}\label{cor:functors}
For each contravariant functor $F$ for which $F(\iota_r): F(BU\times \Z) \to F(B U(r))$ is a surjection, $F(e): F(B \cQ(r)) \to F(B U(r))$ is an injection, and for each covariant functor $G$ for which $G(\iota_r): G(B U(r)) \to G(BU\times \Z)$ is an injection, $G(e): G(B U(r)) \to G(B \cQ(r))$ is an injection.
\end{cor}

\begin{rmk}
Examples of such contravariant functors are given by $B \mapsto H^*(B; \Z)$ \cite[Theorem 16.10, Corollary 16.11]{SwitzerBook} or $B \mapsto K^*(B)$ \cite[Theorem 16.32]{SwitzerBook}, \cite[Lemma 4.3]{AdamsGeneralizedHomologyBook}, and examples of such covariant functors are given by $B \mapsto H_*(B; \Z)$ \cite[Proof of Theorem 16.17]{SwitzerBook} or $B \mapsto \pi_k(B) \otimes \Q,$ \cite[Section 14.7]{GriffithsMorganBook}.
\end{rmk}
\begin{rmk} This generalizes a theorem of Spacil \cite{SpacilThesis}, based on Reznikov-type \cite{citeReznikovCharacteristicclassesinsymplectictopology} characteristic classes obtained by infinite-dimensional Chern-Weil theory for the strict contactomorphism group, stating that $e^*: H^*(B \cQ(r); \R) \to H^*(B U(r); \R)$ is a surjection.
\end{rmk}

We note that Theorem \ref{thm:example} may be interpreted as saying that the map $e$ ``almost''
admits a weak retraction (i.e. a left inverse in the homotopy category).  We therefore end this section with the following question, which we find interesting.


\begin{qtn} Does the map $e: BU (r) \to B \cQ (r)$ admit a weak
   retraction?
\end{qtn}

Finally, we deduce Theorem \ref{thm:mini-example} from Theorem \ref{thm:example}.

\begin{proof}[Proof of Theorem \ref{thm:mini-example}]
First of all Theorem \ref{thm:example} immediately implies that if a complex vector bundle $E$ is not stably trivial, then $\wh{\bP(E)}$ is non-trivial as a prequantum fibration.

The case of $\bP(E)$ as a Hamiltonian fibration is similar, with the difference that a Hamiltonian fibration with vanishing Dixmier-Douady class is trivial if and only if any of its prequantum lifts is obtained from a line bundle over the base of the fibration (and hence every such lift is obtained from a line bundle). This is a direct consequence of the central extension \[0 \to S^1 \to \cQ(r) \to \Ham(\C P^{r-1}) \to 1.\] This shows that Theorem \ref{thm:mini-example} follows from Theorem \ref{thm:example}. 
\end{proof}

\begin{rmk}
Theorem \ref{thm:mini-example} is also a direct consquence of Theorem \ref{thm:twistedmain}.
\end{rmk}



\section{Examples of non-trivial fibrations} \label{section:examples}
Recall that $\cQ (r) = \cQ (\widehat{\mathbb{CP}} ^{r-1} ).$ Using Theorem \ref{thm:mini-example}, and its analogue for prequantum fibrations (that are special cases of Theorem \ref{thm:example} and Theorem \ref{thm:twistedmain}) we provide a new example of a non-trivial $U(r)$ vector bundle that remains non-trivial as a $\cQ(r)$-fibration, and a new example of a vector bundle with non-trivial projectivization that remains non-trivial as a Hamiltonian fibration. These examples are not detected by Reznikov-type classes, vertical Chern classes or the coupling class (see e.g. \cite[Section 2.3]{McDuffLectures}). 



These examples come from the fact that there are stably non-trivial complex vector bundles all of whose Chern classes vanish. Hence their non-triviality is detected by their class in $K$-theory, which by Theorem \ref{thm:example} persists when passing to their isomorphism class as prequantum fibrations, while it is not detected by characteristic classes. For instance, we present the following two concrete examples.



\begin{example}
Let $X = \R P^6$ or $\R P^7.$ Let $\tau_\R$ be the tautological real line bundle over $B = X.$ Let $L = \tau_\R \otimes_\R \C$ be the complexification of $\tau_\R.$ Note that $c:=c_1(L)$ is the non-trivial element in $H^2(B) \cong \Z/2\Z.$ Put $H =L -1$ in complex $K$-theory. It is a well-known computation \cite[Corollary 6.47]{KaroubiBook} that the reduced $K$-theory $\til{K}^0(B) \cong \Z/8\Z,$ with generator $H.$ Hence the rank $4$ complex vector bundle $E = L \otimes \C^4$ over $B$ is stably non-trivial. Hence this provides an example of a non-trivial $\cQ(r)$-fibration. One the other hand, $c(E) = (1 + c)^4 = 1 + 4 c + 6 c^2 + 4 c^3 + c^4 = 1 + c^4 = 1,$ since $2c = 0$ and $H^8(B) = 0.$ Therefore this example is new. Note that $\bP(E) = \bP(\C ^4)$ is trivial, by definition of $E.$
\end{example}

\begin{example}
For the case of a non-trivial $\Ham(\C P^{r-1})$-fibration we modify the first example as follows. Let the base be $B = X \times X.$ Denote $E_1 = \pi_1^* E,$ $E_2 = \pi_2^* E,$ $L_1 = \pi_1^* L,$ $L_2 = \pi_2^* L.$ Let the vector bundle be $E = E_1 \oplus E_2.$ We claim that the projectivization $\bP(E)$ of $E$ is non-trivial. The group $Pic(B) \cong H^2(B,\Z)$ of isomorphism classes of line bundles on $B$ is isomorphic to $\Z/2\Z \times \Z/2\Z$ with generators $L_1,L_2.$ It is easy to verify that $U \otimes E$ is stably non-trivial for all $U \in Pic(B),$ whence the class of $U \otimes [E] - 8$ in reduced $K$-theory is non-zero for all $U \in Pic(B).$ By Theorem \ref{thm:twistedmain} this implies that $\bP (E)$ is non-trivial as a $\Ham(\C P^7)$-fibration. 

Finally, the definition of Chern classes via the Leray-Hirsch theorem (cf. \cite[Proof of Theorem 16.2]{SwitzerBook}) going back to Grothendieck \cite{GrothendieckChern} rests on the fact that, denoting by $u = c_1(T^*)$ the first Chern class of the dual tautological line bundle $T^* = \cO(1)_{\bP (E)}$ over $\bP (E),$ we
have $u^{r} + c_1(\til{E}) u^{r-1} + \dots + c_{r-1}(\til{E}) u +
c_r(\til{E}) = 0.$  Here $\til{E} = \pi_B^* E$ for the projection $\bP(E) \xrightarrow{\pi_B} B.$ Hence $u^{r} = 0,$ and moreover $u$ restricts to the class of the symplectic form on the fibers of $\bP(E) \xrightarrow{\pi_B} B.$ Hence it is the coupling class of the Hamiltonian fibration. Moreover, we note that the vertical tangent bundle $T^{vert} \bP(E)$ of the fibration $\bP(E) \xrightarrow{\pi_B} B$ is isomorphic to $Hom(T,\til{E}/T) \cong T^* \otimes \til{E}/T$ and hence its Chern classes are calculated via $u$ and $c(E) = 1.$ Therefore, the non-triviality of this fibration is not detected via (fiber-integrals) of polynomials in the vertical Chern and coupling classes.
\end{example}

\section{Super twisted K-theory and  topology of $B \text{Ham} (\mathbb{CP}^{r-1} , \omega)$.} \label{sect:super-twisted}
In this section we develop invariants of Hamiltonian fibrations that have values in twisted K-theory, and show that the natural map  $k: BU (r) \to B \text{Ham} (\mathbb{CP} ^{r-1}
, \omega)$, obtained by projectivizing the dual bundle,  is injective on homotopy groups in the ``stable range''. 
This is originially proved for rational homology (and rational homotopy) by Reznikov in
\cite{citeReznikovCharacteristicclassesinsymplectictopology} using infinite
dimensional Chern-Weil theory. It also has a proof via Gromov-Witten
theory, see first author's
\cite{citeSavelyevBottperiodicityandstablequantumclasses.}. We now give a proof
via the techniques of this paper. We remark that the injection holds still for all $k > 1,$ by the work of Casals-Spacil. We outline the implication of their work in the end of this section.


\subsection {The super twisted Dirac K-theory class}
In the previous sections we constructed a natural Atiyah-Singer family over $B
\cQ (\widehat{M} )$. We show that there is likewise a twisted or projective
Atiyah-Singer family over $B \ham$, when $M$ has a prequantum lift
$\widehat{M}= (M, L, \nabla^L,  \omega)$. We associate to this
family a super twisted K-theory class - the Dirac class, which we then show to be a
meaningful new invariant.  We refer to
\cite{AtiyahSegalTwistedK,AtiyahSegalTwistedK2} for more information
about twisted K-theory. The super analogue here is a certain simple graded
variant of this construction, which is completely geometrically natural
from our setup. While this is not necessary for our applications, we then show that in the particular setup of this
paper, this super twisted K-theory and the Dirac class we define can be described using classical twisted K-theory, a more classical invariant.

Using the super twisted K-theory version of our quantization method we then obtain the following variant of the main theorem for the group
$\hamcp$.

\begin{thm} \label{thm:twistedmain}

Let $k: BU(n) \to B \hamcp$ be given by the natural map $U(n) \to \hamcp,$ and $k_*:[B,BU(n)] \to [B, B\hamcp]$ be the induced map, for a topological space $B.$ Consider \[\image k_* \subset [B, B \hamcp].\] There is a ``twisted
quantization'' map $twq$ and a commutative diagram:
\begin{equation*}
   \xymatrix{[B, BU (r)]/ Pic (B) \ar [dr] ^{(W_r \circ i _{r}) _{*} }  \ar [r] & \image k_*  \ar [d]^{twq}\\ 
   & [B, \Fred_r (\cH)]/ Pic (B)},
\end{equation*}
where the notation  $[B, BU (r)]/ Pic (B)$, means the quotient of $[B, BU (r)]$ as a {\em set} by the (natural) action of the topological Picard group $Pic (B)$ of $B.$
\end{thm}

\begin{rmk}
Note that $[B, BU (r)]/ Pic (B) \cong [B, BPU(r)]_{\eta=0},$ the set of isomorphism classes of $PU(r)$-bundles with vanishing Dixmier-Douady class. Guided by this observation and Proposition \ref{Proposition:DD class torsion}, it would be interesting to formulate and prove an extension of Theorem \ref{thm:twistedmain} to the case of non-trivial Dixmier-Douady classes. 
\end{rmk}

\begin{cor}
  The  map $k: BU (r) \to B \hamcp$ is injective on $\pi _{n} $ in
  the range, $2 < n \leq 2r$. 
\end{cor}
\begin{proof} 
   Suppose the class of $f: S ^{n} \to BU (r) $ is non-trivial in
   homotopy. For $n > 2$ the Picard group of $S ^{n} $ is trivial.
   Consequently  $twq (k_* [f]) $ is
   non-trivial if $(W_r \circ i _{r}) _{*} [f]  $ is
   non-trivial, but this is the case since we are in the stable range
   $n \leq 2r$. Recall that the stable range is characterized by the condition that the natural map $\pi_n BU(r) \to \pi_n BU$ is an isomorphism. This can be seen to hold true whenever $n \leq 2r$ by \cite[Lemma 23.4]{citeMilnorMorsetheory}, and the natural isomorphism $\pi_n BU(r) \cong \pi_{n-1} U(r)$ (for $n>1$). 
\end{proof}


\begin{rmk}
We note that our method does not give injectivity on $\pi_n$ outside the stable range $n \leq 2r,$ since for example $\pi_{4r+1}(BU(r)) \cong \Z/(r!)\Z,$ whereas $\pi_{4r+1}(BU) \cong \til{K}_0(S^{4r+1}) = 0.$ However, this injectivity holds for all $n,$ as can be seen from \cite{CasalsSpacil}, by a brief diagram chase from the pair of compatible group extensions (as locally trivial fibrations) \[0 \to S^1 \to \cQ(r) \to \hamcp \to 1,\] and \[0 \to S^1 \to U(r) \to PU(r) \to 1.\]
\end{rmk}

%

Observe that by \eqref{eqn:Quant-Ham-extension} each element $\phi \in \ham$ can be lifted to an automorphism $\wh{\phi}$ of $(L, \nabla^{L} )$ and this lift is unique up to the action of $S^1$. Let 
$
\mathcal{H} ^{gr} (J)$ denote the $\mathbb{Z}_{2} $ graded Hilbert space given by the direct sum of appropriate completions of the Frechet smooth section spaces $\Gamma (L 
\otimes E ^{\pm} , M)$. 
Then $\phi$ induces a well defined element 
\begin{equation} \label{eq-Ham-action}
iso _{\phi}  \in P
U ^{gr} (\cH^ {gr} (J), \cH ^{gr}
(\phi _{*} J )), \end{equation}  the projectivization of the space of graded unitary isomorphisms.
Indeed, this action is induced by the unitary bundle automorphism $D\phi: (TM,J) \to (TM,\phi_* J)$ (not over $id,$ but over $\phi$), and the (well-defined up to the $S^1$-action) map of sections $\Gamma(L) \to \Gamma(L)$ given by $s \mapsto \wh{\phi} \circ s \circ \phi^{-1}.$

Let  ${E}  \ham$ be  the 
universal principal  $\ham $ bundle over $B \ham$. 
Note that a $\ham$ equivariant map $E \ham \to \mathcal{J}$ is equivalent to a family $\{J _{b} \}$ of almost complex structures, $J _{b} $ an almost complex structure on $M _{b}.$ Since $\mathcal{J}$ is contractible, the bundle $E \ham \times_{\ham} \cJ$ has contractible fibers and therefore has a global section. Let \[m:E \ham \to \mathcal{J}\] be the correspoding equivariant map. Denote $m (\phi _{b} )$ by $J_ {\phi _{b}}.$ 

We define Hilbert space bundles $\widetilde{H}^{\pm}  $ over $E\ham
$, with fiber over $\phi _{b} $: $\cH ^{\pm}  (J _{\phi _{b} })$,
defined analogously to $\cH ^{\pm}  (J)$. In other words denoting by $H^{\pm}$ the Hilbert bundles
over $\cJ$ with fiber  $\cH ^{\pm}  (J)$ over $J,$ \begin{equation}\label{eq-tilde-H}\widetilde{H}^\pm = m^* H^{\pm}.\end{equation}

Note that we have a corresponding family of Dirac operators $\{Dir
   ^{1}  (\widehat{M},  J _{\phi _{b} }  )\}$, which we abbreviate by $\{Dir ^{1}
(\widetilde{\cH} _{b}  )\} $.
Set $\widetilde{E} $ to be the graded projective frame bundle of
$\widetilde{H} ^{gr} = \til{H}^+ \oplus \til{H}^- $, that is the fiber of $\widetilde{E} $ over
$\phi _{b} $ is the projectivization $$PU ^{gr} (\cH ^{gr}, \widetilde{H} ^{gr}
_{b}),$$ of the space of graded unitary isomorphisms,  where $\cH
^{gr}$ is a fixed $\Z_2$-graded, infinite-dimensional separable Hilbert space (without loss of generality we assume $\cH^{gr} = \cH^+ \oplus \cH^-,$ with $\cH^+ = \cH^- = \cH$ a fixed separable infinite-dimensional Hilbert space).

We define a $PU ^{gr}  (\cH ^{gr} )$ equivariant map:
\begin{equation*}
\widetilde{Dir}: \widetilde{E} \to  \Fred^{-} (\cH ^{gr}  ),
\end{equation*}
the latter denoting the space of odd Fredholm operators on
$\cH^{gr}$,
by $$\widetilde{Dir} (e _{\phi _{b} } ) =  \widehat {e _{\phi _{b} }}
^{-1}  Dir
^{1}  (\widetilde{\cH} _{b}  ) \widehat {e _{\phi _{b} }}, $$ with
$\widehat {e _{\phi _{b} }}$ denoting any chosen lift of $e _{\phi
_{b}}
      $ to a graded unitary isomorphism, and
where the  $PU^{gr} (\cH ^{gr} )$ action on $\Fred^{-}(\cH ^{gr} )$ is the action:
\begin{equation*}
P \cdot g = \widehat{g} ^{-1}  \circ P \circ \widehat{g},
\end{equation*}
for any chosen lift $\widehat{g} $ of an element of $PU^{gr}(\cH ^{gr} )$, to a
graded Hilbert space isomorphism.

Now we note that $\ham$ freely acts on $\widetilde{E} $ on the left
by $PU ^{gr}(\cH ^{gr} )$ principal bundle automorphisms (not necessarily over $\id$). Explicitly for each $g \in \ham$
we have a graded Hilbert space map  $$\widetilde{H} ^{gr}  _{\phi _{b}} = \cH^{gr} (J
_{\phi _{b} }  ) \to \widetilde{H} ^{gr}  _{\phi _{b} \cdot g ^{-1} }=
\cH^{gr}  ((g ^{-1}) ^{*}  J _{\phi
_{b} } )  $$ well defined up to the action of $S ^{1} $, that is a well defined graded projective Hilbert space map. And moreover $\widetilde{Dir} $ is invariant under this action. Passing to the quotient we get an induced principal $PU ^{gr}  (\cH ^{gr} )$ bundle $\mathcal{P}$ over $B \ham$ and a $PU^{gr} (\cH ^{gr} )$-equivariant map:
\begin{equation*}
   m_{Dir}: \mathcal{P} \to \Fred ^{-}  (\cH ^{gr} ).
\end{equation*}

It is equivalent to a section \begin{equation*}
   Dir \in \Gamma (\mathcal{P} \times_{PU ^{gr} (\cH  ^{gr} )} \Fred ^{-}  (\cH
   ^{gr} )).
\end{equation*}

The homotopy class $[Dir]$ is our \emph{super twisted Dirac K-theory
class}.

\subsection{The splitting isomorphism}

While this is not strictly necessary for subsequent applications, we show here that the above invariant with values in super twisted K-theory can be fully interpreted in terms of classical twisted K-theory. Recall the bundles $\til{H}^+,\til{H}^-$ over $E\ham$ from \eqref{eq-tilde-H}. 

\begin{prop}\label{prop:super via classical}
There is a homotopy-canonical isomorphism \[{s}:\Gamma (\mathcal{P} \times_{PU ^{gr} (\cH  ^{gr} )} \Fred ^{-}  (\cH^{gr} )) \to \Gamma (\mathcal{P}^+ \times_{PU (\cH  ^{+} )} \Fred (\cH^{+} )) \times \Gamma (\mathcal{P}^- \times_{PU (\cH  ^{-} )} \Fred(\cH^{-} )),\] where $\mathcal{P}^+$ and $\mathcal{P}^-$ are the quotients by the $\ham$-action of the projective frame bundles $PU(\cH^+,\til{H}^+)$ and $PU(\cH^-,\til{H}^-),$ and $\mathcal{P}$ is the quotient by the $\ham$-action of the graded projective frame bundle $\til{E} = PU(\cH^{gr},\til{H}^{gr}).$ Moreover, there is a homotopy-canonical isomorphism $\Proj(iso): \mathcal{P}^+ \to \mathcal{P}^-$ under which ${s}(Dir) = (Dir^+,Dir^-)$ takes the adjoint-diagonal form, that is $(Dir^-)^* = \Proj(iso)_* \, Dir^+.$
\end{prop}

We start with a key technical statement. First we note that $\ham$ acts on the bundle $U(\til{H}^+,\til{H}^-)$ over $E\ham$ of unitary bundle-maps $\til{H}^+ \to \til{H}^-.$ This induces a $\ham$-action on the space $\Gamma(U(\til{H}^+,\til{H}^-))$ of sections of $U(\til{H}^+,\til{H}^-).$ Indeed $\cQ(\wh{M})$ acts on both $\til{H}^+$ and $\til{H}^-$ and the $S^1$-ambiguity in the lift disappears once we act on such unitary maps, whence the action descends to $\ham.$ We say that an isomorphism $iso:\til{H}^+ \to \til{H}^-$ is $\ham$-equivariant if the corresponding section $iso \in \Gamma(U(\til{H}^+,\til{H}^-))$ is invariant under the $\ham$-action.

\medskip

\begin{lma}\label{lma-U-lemma}
There is a homotopy-canonical $\ham$-equivariant isomorphism $iso:\til{H}^+ \to \til{H}^-.$ 
\end{lma}

\begin{proof}
First note that $\overline{U} = U(\til{H}^+,\til{H}^-)/\ham$ is a fibration over $B\ham$ with fiber homeomorphic to $U(\cH),$ which by Kuiper's theorem is contractible. Indeed the restriction to each fiber $U(\til{H}^+,\til{H}^-)_e$ of $U(\til{H}^+,\til{H}^-)$ over $e \in E\ham$ of the projection map \[P: U(\til{H}^+,\til{H}^-) \to U(\til{H}^+,\til{H}^-)/\ham\] covering the projection $p: E\ham \to B\ham$ is a homeomorphism \[P_e: U(\til{H}^+,\til{H}^-)_e \to \overline{U}_{p(e)}.\] Choose a section $\overline{iso}$ of $\overline{U},$ which is homotopy-canonical, because the fiber of $\overline{U}$ is contractible. Then we define a $\ham$-invariant section $iso$ of $U(\til{H}^+,\til{H}^-)$ by \[iso(e) = (P_e)^{-1} \overline{iso}(p(e)).\]
\end{proof}

Consider the fibrations \begin{equation}\label{eq-fibrations}
\Fred^-(\til{H}^{gr},\til{H}^{gr}),\; \Fred(\til{H}^+,\til{H}^+),\; \Fred(\til{H}^-,\til{H}^-)
\end{equation} over $E\ham.$ Their quotients by $\ham$ are fibrations over $B\ham$ that are in fact isomorphic to \[\mathcal{P} \times_{PU ^{gr} (\cH  ^{gr} )} \Fred ^{-}  (\cH^{gr} ),\; \mathcal{P}^+ \times_{PU (\cH  ^{+} )} \Fred (\cH^{+} ),\; \mathcal{P}^- \times_{PU (\cH  ^{-} )} \Fred (\cH^{-})\] respectively. By the same argument as in the proof of Lemma \ref{lma-U-lemma} the spaces of  global sections of these quotients are isomorphic to the spaces of $\ham$-equivariant sections of the fibrations \eqref{eq-fibrations}. Finally, a $\ham$-equivariant bundle isomorphism \[\til{s}: \Fred^-(\til{H}^{gr},\til{H}^{gr}) \to \Fred(\til{H}^+,\til{H}^+) \times \Fred(\til{H}^-,\til{H}^-)\] is given by post-composition by \[\begin{pmatrix}
 0 & iso^{-1}\\ iso & 0 
\end{pmatrix}.\]

The map $\til{s}$ induces the isomorphism $s.$ Finally, the isomorphism $\Proj(iso)$ is induced immediately by $iso,$ and the adjoint-diagonality, as in the formulation of Proposition \ref{prop:super via classical}, is a direct consequence of the Dirac operator being self-adjoint.

\subsection {The twisted version of Theorem \ref{thm:example}}
Let us now consider the special case $(M, \omega) = (\mathbb{CP} ^{r-1},
\omega) $.
For $k: BU (r) \to B \hamcp$, and $f: B
\to BU (r) $ we would like
to understand the super twisted Dirac $K$-theory class of the pullback by $k \circ
f$ of the universal $\hamcp$ fibration over $B \hamcp$. Call this
pullback by $P _{k \circ f} $, abreviated by $P$ where convenient.
\begin{lma}\label{lemma:supertwisted}
 The super twisted $K$-theory corresponding to a trivializable
principal $PU ^{gr} 
(\mathcal{H} ^{gr} )$ bundle $\mathcal{P} \to B$, is isomorphic to 
  the group of homotopy classes of maps $X \to \Fred ^{-}
(\cH ^{gr} )$, we shall call this $SK   (B) $.
\end{lma}
Unlike (seemingly) the twisted variant, $SK  (B)$ is not any more interesting than
usual K-theory, as there is a canonical isomorphism $s: SK (B) \to K (B) \times K
(B)$, this is obtained as follows. Fix an isomorphism $iso: \cH ^{+}
\to \cH ^{-}  $, then for $[\sigma] \in SK (B)$, define $$s
([\sigma])= ([iso ^{-1} \circ \sigma| _{\cH ^{+} } ], [iso \circ \sigma | _{\cH
^{-} } ]).$$ Let us denote by $s_+$  the first component of $s$.
\begin{proof}[Proof of Lemma \ref{lemma:supertwisted}]
A principal $PU ^{gr} (\cH ^{gr} ) $ bundle $\mathcal{P} _{0} $ is trivializable if and only if its structure group can be lifted
 to $U ^{gr} (\mathcal{H} ^{gr} )$, since the obstruction to such a
 lift is given by the Dixmier-Douady
  class in $H ^{3} (B, \mathbb{Z}) $ classifying $\mathcal{P} _{0
  } $ as a
  $PU ^{gr} 
 (\mathcal{H} ^{gr} )$ bundle. To lift means to construct an epimorphism of principal bundles:
 \begin{equation*}
    mor:\widetilde{\mathcal{P}} _{0}  \to \mathcal{P} _{0} ,  
 \end{equation*}
with the structure group of $\widetilde{\mathcal{P}} _{0}$: $U ^{gr}  (\mathcal{H}
^{gr} )$. In
other words a bundle map $mor$ satisfying the equivariance condition $mor(p \cdot g) = mor (p) \cdot
cov(g)$, for $cov: U ^{gr} (\mathcal{H} ^{gr} ) \to PU ^{gr}  (\cH
^{gr} )$, the canonical map.

Such a lift canonically determines a lift of
a section of 
\begin{equation*}
\mathcal{P} _{0}  \times _{PU ^{gr}  (\mathcal{H} ^{gr} )} \Fred ^{-}  (\mathcal{H}^{gr} ),
\end{equation*}
to a section of 
\begin{equation*}
\widetilde{ \mathcal{P}  } _{0}\times _{U ^{gr}  (\mathcal{H} ^{gr} )} \Fred^{-} (\mathcal{H} ^{gr} ),
\end{equation*}
and since $\widetilde{\mathcal{P}} _{0}  $ is homotopy canonically
trivializable (by Kuiper's theorem),
this determines a well  defined class in $SK (B)$.
\end{proof}
By construction the structure group of $P$ has a reduction to the
subgroup $\underline {k} (U ({r}) ) \simeq PU (r)$ for  
$\underline {k}: U (r) \to \hamcp$ the natural homomorphism. We denote
the principal $\hamcp$  bundle associated to $P$ by $EP.$  
So reduction of structure group means that we have a morphism (which
is an embedding) of
principal bundles:
\begin{equation*}
   emb: \underline {EP} \to EP,
\end{equation*}
with $\underline {EP}$ a principal $PU (r)$ bundle, with respect to the
natural homomorphism $PU (r) \to \hamcp$.
Since we have an equivariant map $\underline {EP} \to \mathcal{J}$,
given by the constant map to $J_0$ the standard integrable $PU (r)$
invariant complex structure, on $ \mathbb{CP} ^{r-1} $, 
this may be extended (uniquely by equivariance) to a map $m: EP \to \cJ$, so that its
restriction to image of $\underline{E P} $ is the constant map to
$J_0$. 

We identify for convenience $\cH ^{gr} $ with $\cH ^{gr} (J_0)$.
So $$\mathcal{P} _{f}  \equiv  (k \circ f) ^{*} \mathcal{P} \simeq (emb
^{*} (k \circ f) ^{*} \widetilde{E})/PU (r), $$ which has obvious
lifts to $U ^{gr}  (\cH ^{gr} )$ bundles determined by lifts of
$\underline {EP}$
to $U (r)$ bundles. The simplest way to see this is to note that
$\mathcal{P} _{f}$ is exactly the principal $PU ^{gr}(\cH ^{gr} ) $
bundle associated to $\underline {EP}$, and the natural homomorphism $hom: PU (r) \to PU ^{gr}(\cH ^{gr}
) $, and since the latter is covered by $\widetilde{hom}: U (r) \to U ^{gr} (\cH ^{gr} ) $, we
get that for 
 an equivariant bundle map
$\widehat{EP} \to \underline{EP}  $, with $\widehat{EP} $ a $U (r)$ bundle, the 
$U ^{gr} (\cH ^{gr} ) $ bundle associated to $\widehat{EP} $ and the
homomorphism $\widetilde{hom} $, covers $\mathcal{P}_{f} $.


 In
general there is an action of the automorphism group of the trival
$PU ^{gr} (\cH ^{gr} )$ bundle over $B$, $Aut (PU ^{gr}  (\cH ^{gr} ) \times B)$, on the set of possible lifts. 
And $\pi _{0} Aut (PU ^{gr}  (\cH ^{gr} ) \times B) \simeq Pic (B)  $, the Picard group of
$B$. While this is not crucial, in our particular case, it is  possible to see the action of
the Picard group more directly, simply because the Picard group of $B$ acts
on the set of lifts of $\underline{EP}$ to $U (r)$ bundles by tensoring with the
line bundles.


Recall that we denote the univeral $\widehat{\C P} ^{r-1}  $ bundle over $B \cQ$
by $\widehat{P} _{B \cQ}  $. It follows by a straightforward
calculation from the above
discussion that  $s _{+} ([Dir (k \circ f)])$,
for $[Dir (k \circ f)]$ the pullback of super twisted Dirac $K$-theory
class of $P _{B\hamcp} $ by $k \circ f$, can be naturally identified
with the  class $H ^{1}(\widehat{P} 
_{e \circ f}) $ of $ \widehat{P} 
_{e \circ f} = (e \circ f) ^{*} \widehat{P} _{B \cQ}  $, up to the action of the Picard group of $B$,
on
$K$-theory. And by Theorem \ref{thm:example} $H ^{1}(\widehat{P} 
_{e \circ f}) $ is 
just the $K$-theory class of the vector bundle $E _{f} $, classified
by $f$.
\begin{proof} [Proof of Thereom \ref{thm:twistedmain}]
Set $twq$ to be the map taking $[k \circ f]$ to $s_+ ([Dir   (k \circ
f)])$, then our theorem follows immediately by the discussion above.
\end{proof}


\section{Discussion}\label{section:discussion}
It would be very interesting to generalize the $K$-theory quantization to strict contactomorphism groups $\Cont_0(N,\al)$ where the Reeb field $R=R_\al$ defined by $\iota_R d\al = 0, \iota_R \al = 1$ no longer generates a free $S^1$-action, that is when $(N,\al)$ is not a prequantization space. Moreover, Casals-Spacil show that for the standard contact sphere $(S^{2r-1},\xi_{st}),$ there is an injection of homotopy groups $\pi_*(U(r)) \to \pi_* (\Cont_0(S^{2r-1},\xi_{st})),$ of the homotopy of $U(n)$ to the homotopy of the full contactomorphism group $\Cont_0(S^{2r-1},\xi_{st}).$ Hence it would be very interesting to try extending the $K$-theory quantization to full contactomorphism groups and showing an analogue of Theorem \ref{thm:example} for this case. To this end, and in its own right, it would be interesting to see how the constructions of this paper can be cast in terms of operator algebras - cf. \cite{MathaiMelroseSinger0,MathaiMelroseSinger1,MathaiMelroseSinger2} and \cite{ConnesSkandalis}. 

\begin{rmk}\label{Remark: Actual vector bundle}
It was shown by Ma-Marinescu that $\ker (D^k_{-})_{J_x} \cong \mathrm{coker} (D^k_{+})_{J_x} = 0$ for all $k > C(J_x),$ where the function $C$ of the almost complex structure is continuous in the $C^\infty$-topology. This implies that whenever the base $B$ is compact, for all $k$ large enough, the class $H^k$ in $K$-theory is in fact represented by a genuine vector bundle $V^k$ over $B$ with fiber $\ker (D^k_{+})_{J_x}$ over $x \in B.$
\end{rmk}
\begin{rmk}
In fact, whenever $B$ is a closed smooth manifold, given a choice of a smooth section $J \in \Gamma(\cJ_{\cM}, B),$ a prequantum connection $\cA$ on a prequantum fibration\footnote{Here we use the alternative description using principal $S^1$-bundles and connection forms.} $\cP$ over $B$, namely one with parallel transport preserving the contact one-forms on the fibers of $\cP \to B$, is expected to induce a sequence of unitary connections $A^k$ on the vector bundles $V^k.$ Moreover, using the theory of Fourier integral operators of Hermite type, one expects to prove the correspondence principles of the following type. The authors plan to carry this out in detail elsewhere.
\medskip
\begin{enumerate}
\item $\frac{1}{k}||R(A^k)||_\infty \to ||R(\cA)||_\infty$ as $k \to \infty,$ where $R$ denotes the curvature, and the norms are supremum norms over $T^1 B \times T^1 B$ (the square of the unit tangent bundle of $B$ with respect to an auxiliary Riemannian metric on $B$) of natural conjugation invariant norms on the Lie algebra - the operator norm on $Lie(U(n))$ and the $C^0$-norm on $Lie (\Cont(\cP_x,\alpha_x)) \cong \sm{\cM_x,\R}.$

\item $\displaystyle[\trace (e^{t R(A^k)/(k \rank V^k)})]_m \to [\int_{fiber} (e^{t R(\cA)/\vol(M,\om)})]_m,$ as $k \to \infty,$ where $m$ denotes the coefficient of $t^m$ in the power series, which is a $2m$-form on $B$ that is closed by Chern-Weil theory (cf. \cite{SpacilThesis} for the appropriate infinite-dimensional setting). 
\end{enumerate}

In particular we obtain the convergence of certain expressions in the real (and hence rational) Chern classes of $V^k \to B$ to the Reznikov-Spacil classes of $\cP \to B.$ We note that using the Atiyah-Singer index theorem for families, one can compute the Chern character of $V^k,$ and taking limits as $k \to \infty$ obtain new expressions for the Reznikov-Spacil classes. In the special case when $B = S^2,$ isomorphism classes of prequantum fibrations with fiber $(P,\al)$ over $B$ correspond to free homotopy classes of loops in $\Cont(P,\al),$ which in turn correspond to $\pi_1 \Cont(P,\al),$ because this fundamental group is abelian, and the Reznikov-Spacil class $\int R(\cA) / \vol(M, \om)$ corresponds to the Calabi-Weinstein invariant \cite{WeinsteinCalabi}, defined as \[cw(\gamma) = \int_0^1 dt \int_M H_t \om^n,\] where $H_t$ is a contact Hamiltonian of a smooth loop in $\Cont(P,\al)$ representing $\gamma.$ This suggests the following statement, which was not known to the authors, and to the best of their knowledge does not appear in the literature (save for the special case of prequantizations of $\C P^n$ \cite{LoopRemarks}).

\begin{prop}\label{prop: cw quantized}
The Calabi-Weinstein invariant $cw: \pi_1 \Cont(P,\al) \to \R$ satisfies \[cw(\gamma)/\vol(M, \om) = \frac{1}{n+1} \int_{\cM} c^{n+1},\] where $\cM$ is the Hamiltonian fibration underlying any prequantum fibration $\cP \to S^2$ corresponding to $\gamma,$ and $c$ is the first Chern class of the complex line bundle $\cL \to {\cM}.$ In particular, if $\vol(M, \om) = 1,$ then $cw$ takes values in $\frac{1}{n+1} \Z.$
\end{prop}

This fact has a short proof using the usual finite-dimensional Chern-Weil theory for $\cL \to \cM.$ Indeed, one can construct $\cP$ by the clutching construction from two copies $P \times {\D_-}, P \times {\D_+}$ of the trivial fibration $P \times \D$ over the standard disc $\D,$ where a choice of a connection form on $\cP \to \cM$ (or a unitary connection form on $\cL \to \cM$) is constructed from $\al$ and the $1$-periodic contact Hamiltonian $H_t$ of a loop in $\Cont (P,\al)$ representing $\gamma$ as follows (see e.g. \cite{LeonyaActionMaslov}). On $P  \times {\D_-} \to M \times {\D_-}$ take the connection form $\al,$ and on $P  \times {\D_+} \to M \times {\D_+}$ take $\al + d(\rho(r) H_t) \wedge dt,$ where $r$ is the radial coordinate on the disc, $t \in S^1$ is the angular coordinate, and $\rho: [0,1] \to [0,1]$ is a smooth function with $\rho(r) \equiv 0 $ for $r < \epsilon,$ and $\rho(r) \equiv 1$ for $r > 1 -\epsilon,$ where $\epsilon$ is a very small number. Since by Chern-Weil theory this connection form represents $c,$ the calculation is now immediate. 

\end{rmk}

\begin{proof}[Proof of Proposition \ref{Proposition:DD class torsion}]

Using Proposition \ref{prop: cw quantized}, we construct a central extension \[1 \to \mu_{N(P,\al)}  \to
\wh{\Ham}(M,\om) \to \Ham(M,\om) \to 1,\] with  $\wh{\Ham}(M,\om)$ 
a subgroup of $\cQ(P,\al)$, where $\mu_{N(P,\al)} \cong \Z/ N (P,\al)
\Z$ is the group of roots of unity of order $N{(P,\al)} \in \Z_{>0},$
a divisor of $n+1$. Moreover, this extension commutes with the
inclusions $\mu_{N{(P,\al)}} \to S^1$, $\wh{\Ham}(M,\om) \to \cQ(P,\al),$ and the extension \eqref{eqn:Quant-Ham-extension}. From this point the proof is immediate (cf. \cite[Proof of Propostion 2.1 (iv)]{AtiyahSegalTwistedK}). Essentially, the $S^1$-valued cocycle $c_{ijk}$ defining the Dixmier-Douady class can, by means of the new central extension, be chosen to be $\frac{1}{N(P,\al)} \Z / \Z$-valued, hence being $(n+1)$-torsion, and the same is true for its image under the connecting map $H^2(B,S^1) \to H^3(B,\Z).$

To construct $\wh{\Ham}(M,\om)$, first note that the natural Lie algebra
morphism $Lie \Ham (M, \omega) \cong C^\infty_0(M,\R) \to C^\infty(M,\R) \cong Lie \cQ (P,
\al)$ induces a section $i:\til{Ham}(M,\om) \to \til{\cQ}(P,\al)$ to the universal cover level of extension \eqref{eqn:Quant-Ham-extension}: \[0 \to \R \to \til{\cQ}(P,\al) \to \til{Ham}(M,\om) \to 0.\] Define $\wh{\Ham}(M,\om)$ as the image of $\pi \circ i$ where $\pi: \til{\cQ}(P,\al) \to \cQ(P,\al)$ is the canonical endpoint-projection homomorphism. 

It remains to show that the natural surjection $\wh{\Ham}(M,\om) \xrightarrow{pr} \Ham(M,\om)$ has finite kernel. Put $\overline{cw} = cw/\vol(M,\om).$ We construct an injection $0 \to \ker(pr) \xrightarrow{\kappa} \frac{1}{n+1} \Z / \Z = \mu_{n+1}.$ Indeed let $\wh{\phi} \in \ker(pr).$ Then there exists a lift $\til{\wh{\phi}} = i(\gamma) \in \til{\cQ}(P,\al)$ where $\gamma \in \pi_1(\Ham(M,\om)).$ In particular $\overline{cw} (\til{\wh{\phi}}) = 0.$ On the other hand, since $pr(\wh{\phi}) = 1,$ $\wh{\phi} = \phi^\theta_R,$ the Reeb flow for certain time $\theta \in \R.$ Catenating the inverse of the path representing $\til{\wh{\phi}}$ with the path $\{\phi^{t\theta}_R\}_{t=0}^1$ we obtain a class $a$ in $\pi_1 \cQ (P,\al),$ with $\overline{cw}(a) = \theta.$ However by Proposition \ref{prop: cw quantized} this implies that $\theta \in \frac{1}{n+1} \Z.$ It is now immediate to see that the class $\kappa(\hat{\phi}) = [\theta] \in \frac{1}{n+1} \Z/ \Z$ gives a well-defined homomorphism  $\ker(pr) \to \frac{1}{n+1} \Z / \Z,$ which is moreover an injection.

\end{proof}

\section{Methods of infinity-categories}\label{sect:infinity-categories}

In this section we prove Theorem \ref{Proposition: homotopy equivalence commuting}, using methods of infinity-categories to take a suitable homotopy colimit of the maps $p_r: BU(r) \to \Fred_0,$ along the natural maps $BU(r) \to BU(r')$ for $r \leq r'.$

\subsection {Homotopy  colimits} \label{section:hcolim}
For a nice succinct overview of the theory of homotopy (co)limits and model
categories see Riehl \cite{citeRiehlAmodelstructureforquasi-categories}. We shall avoid generalities of model categories and deal purely
with $Top$: the category of paracompact Hausdorff topological spaces.
This theory in effect attempts to solve a pair of not obviously related
problems. First for $D$ a small diagram category, and
\begin{equation*}
   colim: Top ^{D} \to Top,
\end{equation*}
the colimit functor, one would like to have a left total derived functor 
\begin{equation*}
   hocolim: \ho (Top  ^{D})  \to \ho Top
\end{equation*}
for $\ho (Top ^{D}) $ (not to be confused with $(\ho Top)^{D} $) the
homotopy category of $Top ^{D}$: it is just the localization of $Top
^{D} $ with respect to the class of morphisms (here just natural transformations) which
are object wise weak equivalences. Very briefly, as our readers may
not be familiar with these notions, $hocolim$ is the right Kan extension of
$\delta \circ colim$ along $\gamma$ for
$\gamma: Top ^{D} \to \ho Top ^{D}  $, $\delta: Top \to ho Top $ the
localization functors, which in turn means an arrow $R$ in the diagram:
\begin{equation*}
   \xymatrix{Top ^{D} \ar [d] ^{\gamma}  \ar[r]^{colim} & Top \ar [d]
   ^{\delta} \\
\ho (Top ^{D} ) \ar @{-->} [r] ^{R}  &  \ho Top}, 
\end{equation*} 
together with a distinguished natural transformation from $R
\circ \gamma$ to $\delta \circ colim$, which is co-universal in the
natural sense, see also \cite[Section 2]{citeShulman}. 

In this
paper our category $D$ will have as objects the natural numbers, with a
unique morphism $[i,j]$ from $[i]$ to $[j]$ if $j>i$. Again without going into
generalities we point out that if
 $F \in Top ^{D}$ takes all morphisms  in $D$ to cofibrations in $Top$,  then
$$hocolim F = colim F,$$ (for our specific $D$). The specific property of $F$ which ensures this is
being Reedy cofibrant, see Hirschhorn \cite[Chapter 15]{citeHirschhornModelCategories}. 

For the example of this paper, let  
\begin{equation} \label{eq:F}
F: D \to Top
\end{equation}
 be the functor $F ([r])=BU
(r)$, which takes the morphism $[r,r+1]: [r] \to [r+1]$ to the map $BU (r) \to BU
(r+1)$, which is defined 
%
by sending a subspace $H$ to the subspace $H \oplus \mathbb{C} \cdot v _{H} $,
for $\mathbb{C} \cdot v _{H}
$  denoting the subspace generated by $v _{H} $, where $\{v _{H} \}$ is a
chosen continuous family of non-zero elements of $\mathcal{H}$, s.t. each $v
_{H} $ is orthogonal to $H$. Such a family exists since it can be obtained as a
section of a fibration over $BU (r)$, with fiber over $H$ the space of unit
vectors in $\mathcal{H}$ orthogonal to $H$. The fiber is contractible so a
section exists.
 The maps $F ([r,r+1])$ are
cofibrations  and so  in this case
$$hocolim F = BU,$$ as explained above.


The above describes the so called global point of view of the homotopy colimit, and
the output in our specific example. However we may also want a notion of
homotopy colimits, characterized via a homotopical version of the universal
property characterizing the usual colimits, indeed this is what is really
needed in our paper. This is sometimes also called the \emph{local} description
of the homotopy colimit.
In the context of $Top$  the necessary
notion is given by Vogt \cite{citeVogtHomotopyLimitsandColimits}, in terms of homotopy coherent diagrams.
The main point is that these notions of homotopy colimits coincide in our
setting, so that we may compute the local homotopy colimit via the global one.
We shall not state the theorems that imply this claim in absolute generality,  and instead refer the reader to Shulman,
\cite[Sections 8, 10]{citeShulman}. 

While fairly elementary, Vogt's construction is better understood and presented
via the modern theory
of quasi-categories particularly after Joyal and Lurie, simply because
it fits nicely with classical intuition on categories and spaces. This theory is also
 is beautifully documented:
 all of what we
need is contained in the beginning sections of Lurie's
\cite{citeLurieHighertopostheory.}.
Extremely  briefly (without definitions, for which we refer the reader to the
last reference): an $\infty$-category $\mathcal{C}$ is Lurie's abbreviation for an
$(\infty,1)$-category, or more simply a quasi-category, i.e. a simplicial set
satisfying a relaxation of the Kan condition.  From now on
calligraphic letters like $\mathcal{C}$ denote quasi-categories. For $D$ an
ordinary small 
category by a \emph{$D$ shaped diagram} in
$\mathcal{C}$ we mean a simplicial map:
\begin{equation*}
   F: N (D) \to \mathcal{C},
\end{equation*}
where $N (D)$ denotes the simplicial set: nerve of $D$. 
The topological category $Top
$ gives rise to a simplicially enriched category $Top ^{\Delta} $ with morphism
objects $hom _{Top ^{\Delta} } (a,b) = Sing (hom _{Top} (a,b) ) $ for $Sing$ the
singular set functor from the category of spaces to the category of simplicial
sets, and where $hom _{Top} (a,b) $ is the space of maps with the
compact open topology.

The simplicial nerve of $Top ^{\Delta} $, see \cite[Section 1.1.5]{citeLurieHighertopostheory.},
is a
quasi-category we call $\mathcal{T}$. For example its 2-simplices consist of:
topological spaces $T_0, T_1, T_2$, maps $m _{i,j}: T_i \to T_j $, $0 \leq i<j
\leq 2$ and a chosen homotopy from $m _{1,2} \circ m _{0,1} \to m _{0,2}   $,
and so on.
A \emph{co-cone on a $D$ shaped diagram in $\mathcal{C}$} is a simplicial map:
\begin{equation*}
   C _{R}  (N (D)) \to \mathcal{C},
\end{equation*}
naturally extending $F$,
where $C _{R}  (N (D))$ denotes the right cone on the simplicial set $N (D)$, 
i.e. the natural analogue of a topological cone, where there is a 1-edge from
each vertex of $N (D)$ to the cone vertex. See Lurie
\cite [Section 1.2.8]{citeLurieHighertopostheory.} for more details.
The colimit of $F: N (D) \to  \mathcal{C}$ can be interpreted as the universal
co-cone, note however that ``the'' is a little misleading since colimits form
a natural contractible space, rather than being unique. 

\subsection {Proof of Theorem \ref{Proposition: homotopy equivalence commuting}}
Given the setup of the previous section, our strategy is then the following. 
For $D$ and $F: D \to Top$  as in \eqref{eq:F},
we shall first
construct a certain co-cone on $NF: N (D) \to \mathcal{T}$ with cone vertex $c$ mapping to $\Fred _{0}
  (\mathcal{H}) $, and then use that there is a map to this co-cone from the universal co-cone
  with cone vertex $c$ mapping to $BU$, to deduce our claim. The fact that for
  the universal co-cone, the vertex $c$ maps to $BU$ follows by
  \cite[Theorem 4.2.4.1]{citeLurieHighertopostheory.} as it 
   is the
   (local) homotopy colimit of $F$, and so must coincide with (global) homotopy colimit of
  $F$ obtained from the left derived functor, as previously outlined.

For a vertex $[r] \in N (D) $ we map the unique edge $m _{r}:[r] \to c$ to the edge in $\mathcal{T}$ (for $\mathcal{T}$ the quasi-category associated to $Top$ as above) corresponding to the map $p _{r}: BU (r) \to \Fred _{0} (\mathcal{H})$. Then we claim that there is an extension of this to a
map $$L _{\Fred_0 (\mathcal{H})} : C _{R} (N (D)) \to \mathcal{T}.$$ Let us consider 2-simplices of $C _{R}
(N (D) ) $ two of whose edges are $m _{r},m _{r+1}  $. To extend our map to
such a 2-simplex we need to prescribe a homotopy from $$p _{r+1} \circ F([r,
r+1]): BU(r) \to \Fred_0(\cH)$$ to $p _{r}: BU(r) \to \Fred_0(\cH)$. 
For $H \in BU (r)$, let $O _{t} (H) \in \Fred_0(\cH)$ be defined as follows. 
For $v_{H} \in \mathcal{H}$   as in the definition of $F ([r, r+1])$
 the operator $O _{t} (H) \in \Fred_0(\cH)$
is defined to coincide with $p _{r+1} \circ F([r,
r+1]) (H)$ on the orthogonal complement in $\cH$ to $v _{H} $ and taking $v
_{H} $ to the image under $iso_{r+1}: \cH \oplus \C^{r+1} \to \cH$ of $$t 
\cdot (0, e_r) \in \mathcal{H} \oplus \mathbb{C}^{r+1}, $$ for $e _{r} $ any
fixed non zero vector in $\mathbb{C}^{r} $. Clearly $$O _{0} (H)= p _{r+1} \circ F([r,
r+1]) (H),$$ and  $O _{1} (H)$ has the same kernel as $p _{r} (H) $, and
naturally isometric image. Specifically we may find an isometry $$\mathcal{I}
_{r}: \cH \to \cH $$ so that $\mathcal{I}_{r} \circ O_1 (H) $ has the same
image as $p _{r} (H) $ for every $H$. We then define a homotopy 
$${ht} _{t}  = \mathcal{I}_{r,t} \circ O _{t}$$ where $\{\mathcal{I} _{r,t}\} $, $0
\leq t \leq 1$ is any homotopy from $id$ to $\mathcal{I} _{r} $. Then $ht _{1} 
(H)$
has the same kernel and image as $p _{r} (H) $. (By Kuiper's theorem the space
of such homotopies is non-empty and contractible.)
Next set
$\widetilde{ht} _{t}  $ to be any correction continuous in $H$ of the homotopy $ht _{t} $ so that
$$\widetilde{ht}_ {0} = p _{r+1} \circ F ([r, r+1])   $$ and so that 
\begin{equation*}
   \widetilde{ht} _{1} = p _{r}.
\end{equation*}
Such a correction exists since
 by Kuiper's theorem the space of Fredholm
 operators with the same  kernel and image are contractible, and hence the corresponding bundle over $BU(r)$ has contractible fibers.
Using Kuiper's theorem and the fact that the spaces of frames in $\mathcal{H}$ are contractible we may then similarly extend our map to the higher simplices of $C _{R} (N (D))$,
to obtain a map  $$L _{\Fred _{0} (\mathcal{H}) } : C _{R} (N (D)) \to
\mathcal{T}.$$

Passing to the
homotopy categories (of the simplicial sets, as defined in
\cite[Section 1.2.3]{citeLurieHighertopostheory.}) this induces a co-cone that we denote
by $\ho L _{\Fred _{0} (\mathcal{H}) } $ on $$\ho F: D \to \ho \mathcal{T} = \ho
Top$$ with vertex
$\Fred _{0} (\mathcal{H}) $, (this is just the classical co-cone, i.e. a natural
transformation in $\ho Top$ to the constant functor). We also have ``the'' universal co-cone 
$$L _{BU}:  C _{R} (N (D)) \to
\mathcal{T},$$ with the associated classical co-cone $\ho L _{BU}$ on  $D \to
\ho Top. $

Now by definition of colimit as outlined in the previous section and maps of
co-cones, (see \cite [Section 1.2.9]{citeLurieHighertopostheory.}) there is 
a morphism (again in the obvious classical sense) $$\ho L _{BU} \to \ho
L _{\Fred _{0} (\mathcal{H}) }, $$ which in particular induces a map $$W: BU \to \Fred
_{0} ( \mathcal{H}).$$ We need to check that it is a weak equivalence. 
It can be readily shown as for example in Atiyah's book \cite[Proposition A.5]{AtiyahBook} that 
for any $f: B \to \Fred _{0} (\mathcal{H})  $, with $B$ compact there is a homotopy representative
$f'$ so that the kernel spaces $\{\ker f' (s)\}$ form a family of subspaces of
$\mathcal{H}$ of
the same  dimension, and so that the cokernels of $\{f' (s)\}$ are a fixed
subspace of $\mathcal{H}$. Then by construction it readily follows that $W$ is
a isomorphism on homotopy
groups of $\Fred _{0} (\mathcal{H}) $.
\qed

\section*{Acknowledgements}
We are grateful to Leonid Polterovich for introducing us to quantization from the viewpoint of symplectic topology. The work on this paper has started and was carried out mostly during
the authors' stay in CRM, University of Montreal, and also during the
stay of Y.S. at ICMAT, Madrid under funding from Severo Ochoa grant, and the stay of E.S. at the Hebrew
University of Jerusalem, under funding from ERC grant no. 337560, and at Universite Lyon $1$ Claude Bernard, under funding from ERC grant no. 258204. 
We thank these institutions for their warm hospitality. We thank Octav Cornea and Francois Lalonde for their continuous support, and we thank Jake Solomon for inviting Y.S. for a visit to the Hebrew University, which was pivotal to this project. We thank Joseph Bernstein, Dusa McDuff, and Fran Presas for their interest, and Jarek Kedra for useful comments. Finally, we thank the referees for their help with the exposition.

 \bibliographystyle{siam}  
 \bibliography{link} 

\medskip

\end {document}